\documentclass[a4paper,10pt,twoside]{article}
\usepackage{amsmath,amsthm,amsfonts}
\usepackage{mathtools}
\usepackage[pdfborder={0 0 0},pagebackref=false]{hyperref}
\usepackage{geometry}
\usepackage{color}

\title{An ergodic theorem for partially exchangeable random partitions}

\author{Jim Pitman\footnote{Statistics Department, 367 Evans Hall \# 3860, University of California, Berkeley, CA 94720-3860, U.S.A.
    \texttt{pitman@berkeley.edu}}
  \and 
 Yuri Yakubovich\footnote{Saint Petersburg State University, St.\ Petersburg State University, 7/9 Universitetskaya nab., St.\ Petersburg, 199034 Russia. \texttt{y.yakubovich@spbu.ru}}}

\newcommand{\Cl}[1]{ {\mathcal C }_#1}

\newcommand{\Xp} { X}

\newcommand{\mindex}[2]{    { M_{#1, #2} } }

\newcommand{\Pdec}{P^{\downarrow}}
\newcommand{\BN}{\mathbb{N}}

\newcommand{\CC}{ {\mathcal C }}
\newcommand{\Q}{R}
\newcommand{\tilCC}{ \widetilde{ {\mathcal C }} }

\newcommand{\BP}{\mathbb{P}}

\renewcommand{\P}{\mathbb{P}}

\newcommand{\E}{\mathbb{E}}

\newcommand{\bull}{\cdots}

\newcommand{\Pbul}{(P_{j})}
\newcommand{\Pbuln}[1]{\bigl(P_{j}^{(#1)}\bigr)}
\newcommand{\Pbulstar}{(P_{j}^*)}

\newcommand{\Pdecbul}{\bigl(P^\downarrow_{j}\bigr)}

\newcommand{\ed}{\overset{d}{{}={}}}

\newcommand{\tilPi}{\widetilde{\Pi}}
\newcommand{\hatPi}{\widehat{\Pi}}

\newcommand{\giv}{\,|\,}  

\definecolor{gray}{gray}{0.4}
\newtheorem{theorem}{Theorem}
\newtheorem{proposition}[theorem]{Proposition}

\newtheorem{corollary}[theorem]{Corollary}
\theoremstyle{definition}
\newtheorem*{remark*}{Remark}

\numberwithin{equation}{section}
\numberwithin{theorem}{section}
\numberwithin{remark}{section}

\begin{document}



\maketitle

\begin{abstract}
We consider shifts $\Pi_{n,m}$ of a partially exchangeable random partition $\Pi_\infty$ of $\BN$ obtained by restricting $\Pi_\infty$
to $\{n+1,n+2,\dots, n+m\}$ and then subtracting $n$ from each element to get a partition of $[m]:= \{1, \ldots, m \}$. We show that for each fixed $m$ the distribution of
$\Pi_{n,m}$ converges to the distribution of the restriction to $[m]$ of the exchangeable random partition of $\BN$ with the same ranked frequencies as $\Pi_\infty$. As a consequence, the partially
exchangeable random partition $\Pi_\infty$ is exchangeable if and only if  $\Pi_\infty$ is stationary in the sense that for each fixed $m$ the distribution of $\Pi_{n,m}$ on partitions of $[m]$ is the same for all $n$.
We also describe the evolution of the frequencies of a partially exchangeable random partition  under the shift transformation. For an exchangeable
random partition with proper frequencies, the time reversal of this evolution is the {\em heaps process} studied by Donnelly and others.
\end{abstract}

\tableofcontents

\newpage

\section{Introduction}

A random partition $\Pi_\infty$ of the set $\BN$ of positive integers arises naturally in a number of different contexts.
The fields of application include population genetics \cite{MR2026891} 
\cite{MR0526801}, 
statistical physics \cite{Higgs95},  
Bayesian nonparametric statistics 
\cite{MR350949} 
and many others. Moreover, this subject has some purely mathematical interest. We also refer to 
\cite{MR2245368}  
for various results on random partitions.

There are two convenient ways to encode a random partition $\Pi_\infty$ of the set $\BN$ as a sequence whose $n$th term
ranges over a finite set of possible values. One way is to identify $\Pi_\infty$ with its sequence of restrictions $\Pi_n$ to the
sets $[n]:= \{1, \ldots, n \}$, say $\Pi_\infty = (\Pi_n)$ where $n$ will always range over $\BN$. Another encoding is provided by the
{\em allocation sequence} $(A_n)$ with $A_n = j$ iff $n \in \CC_j$ where $\Pi_\infty = \{ \CC_1, \CC_2, \ldots \}$
with the {\em clusters } $\CC_j$ of $\Pi_\infty$ listed in increasing order of their least elements, also called {\em order of appearance}.
Commonly, random partitions $\Pi_\infty$ of $\BN$ are generated by some sequence of random variables $(X_n)$, meaning that 
$(\CC_j)$ is the collection of equivalence classes for the random equivalence relation $m \sim n$ iff $X_m = X_n$. Every random partition of $\BN$
is generated in this way by its own allocation sequence.
If $F$ is a random probability distribution, and given $F$ the sequence $(X_n)$ is i.i.d.\ according to $F$,  and then $\Pi_\infty$ is generated by $(X_n)$,
say $\Pi_\infty$ is {\em generated by sampling from $F$}.
Kingman \cite{MR0526801}  
\cite{MR509954} 
developed a theory of random partitions that are {\em exchangeable} in the sense that for each $n$ the distribution of $\Pi_n$ on the
set of partitions of $[n]$ is invariant under the natural action on these partitions by permutations of $[n]$. Kingman's main results
can be summarized as follows:
\begin{itemize}
\item every exchangeable random partition $\Pi_\infty$ of $\BN$ has the same distribution as one generated by sampling from some random distribution $F$ on the real line;
\item the distribution of $\Pi_\infty$ generated by sampling from $F$ depends only on the joint distribution of the list $\Pdecbul$ of sizes of
atoms of the discrete component of $F$, in weakly decreasing order.
\end{itemize}
Two immediate consequences of these results are:
\begin{itemize}
\item every cluster $\CC_j$ of an exchangeable random partition $\Pi_\infty$  of $\BN$ has an almost sure limiting relative frequency $P_j$;
\item ranking those limiting relative frequencies gives the distribution of ranked atoms $\Pdecbul$ required to replicate
the distribution of $\Pi_\infty$ by random sampling from an $F$ with those ranked atom sizes.
\end{itemize}
Kingman's  method of  analysis of exchangeable random partitions of $\BN$,
 by working with the distribution of its ranked frequencies $(\Pdec_j)$, continues to be used in the study
of partition-valued stochastic processes \cite{MR2596654}.
But well known examples, such as the random partition of $\BN$ whose distribution of $\Pi_n$ is given by the Ewens sampling formula
\cite{MR0325177} \cite{MR1434129}, show it is often more convenient to encode the distribution of an exchangeable random partition of $\BN$ by the distribution of its frequencies of
clusters $\Pbul$ in their order of appearance, rather than in weakly decreasing order. 
This idea was developed in  Pitman \cite{MR1337249}, together with a more convenient encoding of the distribution of $\Pi_n$.
Call $\Pi_\infty$ a {\em  partially exchangeable partition (PEP) of\/ $\BN$} if for each fixed $n$ the distribution of $\Pi_n$ is given by the formula
\begin{equation}
\label{pepf}
\P( \Pi_n = \{ C_1, \ldots, C_k \} ) = p(\#C_1, , \ldots, \#C_k ) 
\end{equation}
for each particular partition of $[n]$ with $k$ clusters $C_1, \ldots, C_k$  in order of appearance, 
of sizes $\#C_1, \ldots, \#C_k$,
for some function $p(n_1, \ldots, n_k)$ of
{\em compositions of $n$}, meaning sequences of positive integers $(n_1, \ldots, n_k)$ with $\sum_{i=1}^k n_i = n$ for some $1 \le k \le n$.
The main results of \cite{MR1337249} can be summarized as follows. See also \cite[Chapters 2,\;3]{MR2245368}. 
\begin{itemize}
\item There is a one-to-one correspondence between distributions of partially exchangeable partitions of $\BN$ and
non-negative functions $p$ of compositions of positive integers subject
to the normalization condition $p(1) = 1$ and the sequence of addition rules
\begin{equation}\label{PEPaddition}
\begin{aligned}
p(n) &= p(n + 1 )  + p(n,1), \\
p(n_1,n_2) &= p(n_1+1, n_2)+  p(n_1 ,n_2+1) + p( n_1 , n_2, 1) 
\end{aligned}
\end{equation}
and so on. 
This function $p$ associated with $\Pi_\infty$ is called its {\em partially exchangeable partition probability function (PEPPF)}.

\item $\Pi_\infty$ is exchangeable iff $\Pi_\infty$  is partially exchangeable with a $p(n_1, \ldots, n_k)$ that is for each fixed $k$ a symmetric function of its $k$ arguments.

\item Every cluster $\CC_j$ of a partially exchangeable random partition $\Pi_\infty$  of $\BN$ has an almost sure limiting relative frequency $P_j$,
with $P_j = 0$ iff $\CC_j$ is a {\em singleton}, meaning $\# \CC_j = 1$.

\item The distribution of the sequence of cluster frequencies $(P_j)$ and the PEPPF $p$ determine each other by the {\em product moment formula}
\begin{equation}
\label{expeppf}
p(n_1, \ldots, n_k ) = \E \prod_{i=1}^k ( 1 - \Q_{i-1}) P_i^{n_i - 1 } \; \mbox{ where }\; \Q_{i}:= \sum_{j = 1}^i P_j
\end{equation}
is the cumulative frequency of the first $i$ clusters of $\Pi_\infty$.

\item The set of all PEPPFs $p$ is a convex set in the space of bounded real-valued functions of compositions of positive integers,  compact in the topology
of pointwise convergence.
\item
The extreme points of this convex compact set of PEPPFs are given by the formula \eqref{expeppf} for non-random sequences of sub-probability cluster frequencies $\Pbul$, meaning that
$P_j \ge 0$ and $\sum_{j} P_j \le 1$.

\item The formula
\eqref{expeppf}, for a random sub-probability distribution $\Pbul$, provides the unique representation of a general PEPPF as an integral mixture of these extreme PEPPFs.

\item The family of distributions of partitions of $\BN$ with PEPPFs \eqref{expeppf}, as a fixed sequence $\Pbul$ varies over all sub-probability distributions, and the $\E$ can be omitted,
provides for every exchangeable or partially exchangeable random partition $\Pi_\infty$ of $\BN$ a regular conditional distribution of $\Pi_\infty$ given its cluster frequencies $\Pbul$
in order of appearance.
\end{itemize}

These results provide a theory of partially exchangeable random partitions of $\BN$ that is both simpler and more general than the theory of exchangeable random partitions.
The structure of partially exchangeable random partitions of $\BN$ is nonetheless very closely tied to that of exchangeable random partitions, due to the last point above.
Starting from the simplest exchangeable random partition of $\BN$ with an infinite number of
clusters, whose cumulative frequencies $(\Q_k)$ have the same distribution as the sequence of record values of an i.i.d.\ uniform $[0,1]$ sequence, given by the stick-breaking
representation
\begin{equation}\label{stickbreaking}
1-\Q_k=\prod_{i=1}^{k} (1-H_i), \qquad k=1,2,\dots,
\end{equation}
for $H_i$ a sequence of i.i.d.\ uniform $[0,1]$ variables, the most general extreme partially exchangeable random partition $\Pi_\infty$ of $\BN$ with fixed cluster frequencies
$\Pbul$ may be regarded as derived from this record model by conditioning its cluster frequencies. 
See \cite{MR1374320} 
for further development of this point.
Less formally, a PEP is as exchangeable as it possibly can be, given that its distribution of cluster frequencies $\Pbul$ in appearance order has been
altered beyond the constraints on the frequencies of an exchangeable random partition of $\BN$. For {\em proper frequencies} $\Pbul$, 
with $\sum_j P_j =1$ almost surely, those constraints are that
\begin{itemize}
\item a partially exchangeable $\Pi_\infty$ with proper frequencies $\Pbul$ is exchangeable iff $\Pbul \ed \Pbulstar$ where $\Pbulstar$ is a size-biased random permutation of $\Pbul$.
\end{itemize}
Then $\Pbul$ is said to be in {\em size-biased random order} or {\em invariant under size-biased random permutation} 
\cite{MR996613} 
\cite{MR1387889}. 

For a partially exchangeable random partition $\Pi_\infty$,
consider for each $n = 0,1,2, \ldots$ the random partition $\Pi_\infty^{(n)}$ of $\BN$ defined by first restricting $\Pi_\infty$ to $\{n+1, n+2, \ldots \}$,
then shifting indices back by $n$ to make a random partition of  $\{1,2, \ldots\}$ instead of $\{n+1,n+2, \ldots\}$. This procedure
appeared in our paper \cite{0gibbs} as discussed further in Section \ref{sec:crl} below.
If $\Pi_\infty$ is exchangeable, then obviously so is $\Pi_\infty^{(n)}$, because $\Pi_\infty^{(n)} \ed \Pi_\infty$ for every $n$.
Moreover, the sequence of random partitions $\bigl(\Pi_\infty^{(n)}, n = 0,1, \ldots\bigr)$ is a stationary random process to which the ergodic theorem can be immediately
applied. According to Kingman's representation, this process of shifts of $\Pi_\infty$ is ergodic iff the ranked frequencies of $\Pi_\infty$ are constant almost surely, 
For more general models with random ranked frequencies, the asymptotic behavior of functionals of $\Pi_\infty^{(n)}$ 
can be read from the ergodic case by conditioning on the ranked frequencies.

If $\Pi_\infty$ is only partially exchangeable, it is easily shown that $\Pi_\infty^{(n)}$ is also partially exchangeable for every $n$.
The PEPPF $p^{(n)}$ of $\Pi_\infty^{(n)}$ is obtained by repeated application of  the following simple transformation from the 
PEPPF $p$ of $\Pi_\infty$ to the PEPPF  $p^{(1)}$ of $\Pi_\infty^{(1)}$:
\begin{equation}\label{PEPshift}
\begin{aligned}
p^{(1)}(n) &= p(1,n) + p(n + 1 ),   \\
p^{(1)}(n_1,n_2) &=   p(1,n_1,n_2) + p( n_1 + 1, n_2) + p(n_1, n_2 + 1) 
\end{aligned}
\end{equation} 
and so on, in parallel to the basic consistency relations \eqref{PEPaddition}
for a PEPPF. 
If $\Pi_\infty$ is partially exchangeable, with 
$\Pi_\infty^{(1)} \ed \Pi_\infty$, or equivalently $p^{(1)}(\bull) = p(\bull)$,
 then call $\Pi_\infty$ {\em stationary}. Its sequence of shifts $(\Pi_\infty^{(n)}, n = 0,1, \ldots)$ is then a stationary random process to which the ergodic theorem can be applied.
That raises two questions: 
\begin{itemize}
\item[(i)] Are there any partially exchangeable random partitions of $\BN$ which are stationary but not exchangeable?
\item[(ii)] If a partially exchangeable random partition of $\BN$ is not stationary, what can be provided as an ergodic theorem governing the long run behavior of its sequence of shifts?
\end{itemize}
Since $\Pi_\infty$ is exchangeable iff $p$ is symmetric, the answer to the question (i) is ``yes'' if and only if
\begin{equation}
\label{funcsym}
\parbox{0.8\textwidth}{
{\em 
every function $p$ of compositions that is bounded between $0$ and $1$ and satisfies both systems of equations \eqref{PEPaddition} and \eqref{PEPshift} is a symmetric function of its arguments.
}}
\end{equation}
So it seems the matter should be resolved by analysis of the combined system of equations. Surprisingly, this does not seem to be 
easy. Still, we claim that every partially exchangeable and stationary random partition of $\BN$ is in fact exchangeable,
so 
\eqref{funcsym} is true.
We do not know how to prove this without dealing with question (ii) first. But
that question is of some independent interest, so we formulate the following theorem:
\begin{theorem}
\label{thm:main}
Let\/ $\Pi_\infty$ be a partially exchangeable random partition of positive integers with ranked frequencies $\Pdecbul$, and let $\Pbuln{n}$ for 
each $n = 0,1, \ldots$ be the frequencies of clusters of\/ $\Pi_\infty$ 
in the order of appearance of these clusters to in\/ $\Pi_\infty^{(n)}$ obtained from the restriction of\/ $\Pi_\infty$ restricted to $\{n+1, n+2, \ldots \}$. Then: 
\begin{itemize}
\item
As $n \to \infty$, the distribution of\/ $\Pi_\infty^{(n)}$ converges weakly to that of the exchangeable random partition $\tilPi_\infty$ of\/ $\BN$ with ranked frequencies $\Pdecbul$, meaning that
the PEPPF $p^{(n)}(\bull)$ 
of\/ $\Pi_\infty^{(n)}$ 
converges pointwise to the EPPF 
$p^{(\infty)}(\bull)$ 
of\/ 
$\tilPi_\infty$.
\item
As $n \to \infty$, the finite dimensional distributions of\/ $\Pbuln{n}$ converge weakly to those of $\Pbul$, the list of frequencies in order of appearance of
an exchangeable random partition of $\BN$ with ranked frequencies $\Pdecbul$, which for proper $\Pdecbul$ with $\sum_j \Pdec_j = 1$ is a size-biased random permutation of $\Pdecbul$,
or of $\Pbuln{n}$ for any fixed $n$.
\item
$\Pi_\infty$ is exchangeable iff the partition-valued process $\Pi_\infty^{(n)}$ is stationary, meaning that $\Pi_\infty^{(n)} \ed \Pi_\infty$ for $n =1$, hence for all $n \ge 1$,
or, equivalently, the PEPPF $p^{(n)}(\bull)$ equals the PEPPF $p(\bull)$ of\/ $\Pi_\infty$ for $n =1$, hence for all $n \ge 1$.
\end{itemize}
\end{theorem}

In view of the one-to-one correspondence between the law of a partially exchangeable partition $\Pi_\infty$ and the law of its frequencies of clusters in order of appearance, this theorem has the following corollary:
\begin{corollary}
\label{crl:main}
In the setting of the previous theorem,
with $\Pbuln{n}$ for each $n = 0,1, \ldots$ the frequencies of a partially exchangeable partition\/ $\Pi_\infty$ 
in their order of appearance when $\Pi_\infty$ is restricted to $\{n+1, n+2, \ldots \}$,
\begin{itemize}
\item
the sequence $\bigl(\Pbuln{n}, n = 0,1, \ldots\bigr)$ is a Markov chain with stationary transition probabilities on the space of sub-probability distributions of $\BN${\upshape;}
for $n\ge1$ the forwards transition mechanism from $\Pbuln{n-1}$ to $\Pbuln{n}$ is by a {\em top to random move}, whereby
given $\Pbuln{n-1} = \Pbul$ the value $\Pbuln{n}$ is either $(P_2,P_3,\ldots)$ if $P_1=0$, or 
$$
(P_2,\dots,P_{X},P_1,P_{X+1},P_{X+2},\dots)
\mbox{ if } P_1>0,
$$ 
for some 
random position $X \in \BN$ with the proper conditional distribution
\begin{equation}
\label{Jdist}
\BP[ X>j \giv \Pbul] = 
\prod_{i=1}^j\Bigl(1-\frac{P_1}{1-P_2-P_3-\dots-P_i}\Bigr),\qquad j=1,2\dots;
\end{equation}
\item
$\Pi_\infty$ is exchangeable if and only if the Markov chain $\Pbuln{n}$ is stationary, meaning that $\Pbuln{n} \ed \Pbuln{0}$ for $n =1$ and hence for all $ n \ge 1$; 
\item
if\/ $\Pi_\infty$ is exchangeable the reversed transition mechanism from $\Pbuln{n}$ to $\Pbuln{n-1}$ is by a {\em random to top move}, whereby
with probability $1-\sum_j P^{(n)}_j$ the frequency $0$ is prepended to the sequence $\Pbuln{n}$, 
otherwise 
with probability $P^{(n)}_{j}$ the frequency $P^{(n)}_{j}$ is removed and put in place $1$.
\end{itemize}
\end{corollary}

According to the last part of the Corollary, when $\Pi_\infty$ is exchangeable, with proper frequencies, the time-reversed random-to-top evolution of the
cluster frequencies $\Pbuln{n}$ of $\Pi_\infty^{(n)}$
is the mechanism of the {\em heaps process} studied by Donnelly \cite{MR1104569}, 
also called a {\em move-to-front rule}.
The mechanism of this chain has been extensively studied, mostly in the case of finite number of nonzero frequencies, due to its interest
in computer science
\cite{MR1603279} 
\cite{MR1700744}. 
Donnelly's result that proper frequencies  $\Pbul$ are in a size-biased order iff the distribution of $\Pbul$ is invariant under this transition mechanism is an immediate consequence of the above corollary.
It seems surprising, but nowhere in Donnelly's article, or elsewhere in the literature we are aware of, is it mentioned that the
random-to-top rule is the universal time-reversed evolution of  cluster frequencies in order of appearance for shifts of any exchangeable random partition of $\BN$ with proper frequencies.
We are also unaware of any previous description of the time-forwards evolution of these cluster frequencies, as detailed in the corollary.

The rest of this article is organized as follows. Theorem \ref{thm:main} is proved in Section \ref{sec:proof}.
In Section \ref{sec:crl} we first recall an idea from \cite{0gibbs} which led us to develop the results of this article. This leads to a proposition
which we combine with Theorem \ref{thm:main} to obtain Corollary \ref{crl:main}. Finally, Section \ref{sec:rel} provides some references to related literature.

\section{Proof of Theorem \ref{thm:main}}
\label{sec:proof}
\begin{proof}
The convergence in distribution of $\Pi_\infty^{(n)}$ to $\tilPi_\infty$ is obtained by a coupling argument.
Given the frequencies $\Pbuln{0}$ of $\Pi_\infty$ and the independent i.i.d.\ sequence $(U_j)$ of uniform on $[0,1]$ random variables, 
let us construct
a partially exchangeable random partition $\hatPi_\infty$ distributed as $\Pi_\infty$, and an exchangeable random partition
$\tilPi_\infty$ such that the convergence of $\hatPi_\infty^{(n)}$ to $\tilPi_\infty$ holds almost surely.
Set $R_k:= \sum_{i=1}^k P^{(0)}_i$ and construct
$\hatPi_\infty$ as the partition generated by values of the table allocation process $(A_n)$ defined by $A_1:= 1$ and given that $A_1, \ldots, A_n$
have been assigned with $K_n:= \max_{1 \le i \le n} A_i$ distinct tables,  $A_{n+1} = j$ if $U_{n+1} \in (R_{j-1},R_j]$ for some $1 \le j \le k$
and $A_{n+1} = k+1$ if $U_{n+1} \in (R_{k},1]$. The limiting exchangeable
random partition $\tilPi_\infty$ is conveniently defined on the same probability space to be the random partition of $\BN$ whose list of clusters with strictly
positive frequencies is $\tilCC_k:= \{n : U_n \in (R_{k-1},R_k] \}$ for $k$ with $R_{k-1} < R_k$, and with each remaining element of $\BN$ a singleton cluster.
Let $\CC_k$ be the $k$th cluster of  $\hatPi_\infty$ in order of appearance. The key observation is that for each $n \ge 1$ the intersections of 
 $\CC_k$ and $\tilCC_k$ with $[n+1, \infty)$ are identical on the event $(K_n \ge k)$. In more detail, if say $K_n = k$, then for all $ i > n$
\begin{itemize}
\item if $U_i \le R_k$ then almost surely both $i \in \CC_j$ and $i \in \tilCC_j$ for some $1 \le j \le k$;
\item if $U_i  > R_\infty:= \lim_n R_k$ then $i \in \CC_j$ for some $j > k$, while $\{i\}$ is a singleton cluster of $\tilPi_\infty$;
\item if $U_i  \in(R_k, R_\infty]$ then $i \in \CC_j$ and $i \in \tilCC_{\ell}$ for some $j > k $ and $\ell >k$.
\end{itemize}
Consider the restrictions $\hatPi_{[n+1,n+m]}$ and $\tilPi_{[n+1,n+m]}$ of
$\hatPi_\infty$ and $\tilPi_{\infty}$ to the interval of integers $[n+1,n+m]$ and call their clusters 
which are non-empty intersections of $\CC_j$ for $1\le j \le k$ with $[n+1,n+m]$ {\em old\/} and all other clusters {\em new}. It follows from the above description that 
old clusters are the same for both partitions, and only new clusters may differ. However it turns out that if $n$ is large and $m$
is fixed, with high probability all new clusters are singletons in both partitions. Indeed, for a new cluster $\CC$ of  $\hatPi_{[n+1,n+m]}$ to contain 
an element $j > i:=\min \CC$, $U_j$ must hit some interval $(R_{\ell-1}, R_\ell]$ with $\ell>k$, and in
particular must hit $(R_k, R_\infty]$. For a new cluster $\tilCC$ of $\tilPi_{[n+1,n+m]}$ to contain two elements $i$ and $j$, $U_i$ and $U_j$ both should get  
into some interval $(R_{\ell-1}, R_\ell]$ with $\ell>k$. Thus there is the coupling bound
\begin{equation}
\label{couping}
\P( \hatPi_{[n+1,n+m]} \ne \tilPi_{[n+1,n+m]} ) \le m\, \E ( R_\infty - R_{K_n} ) .
\end{equation}
But as $n \to \infty$, there is almost sure convergence of $R_{K_n}$ to $R_\infty \le 1$, so the bound converges to $0$ for each fixed $m$. This 
proves pointwise convergence  of the PEPPF of $\Pi_\infty^{(n)}$ to the EPPF of $\tilPi_\infty$. The convergence of 
finite-dimensional distributions of $\Pbuln{n}$ to those of $\Pbul$ follows from \cite[Theorem 15]{MR1337249}.
The final assertion, not obvious only in part that if $\Pi_\infty^{(n)}$ is stationary, then $\Pi_\infty$ is exchangeable, follows immediately.
\end{proof}

\section{Sampling frequencies in size-biased order}
\label{sec:crl}

Let $\Cl{1}, \Cl{2}, \ldots$ be the list of clusters of an exchangeable random partition
$\Pi_\infty$, in the appearance order of their least elements. 
Let $\mindex{i}{1}:=\min \Cl{i}$, and assuming that $\Cl{i}$ is infinite 
let $\mindex{i}{1} < \mindex{i}{2} <  \cdots$ be the elements of $\Cl{i}$ listed in increasing order. So in particular 
$1 = \mindex{1}{1} < \mindex{2}{1} < \cdots$ is the list of least elements of clusters $\Cl{1}, \Cl{2}, \ldots$.
Observe that the number of clusters $K_n$ of $\Pi_n$ is $K_n = \sum_{i = 1}^n 1(\mindex{i}{1} \le n)$. 
Let $\Xp$ 
be the number of 
distinct clusters of 
$\Pi_\infty$, 
including the first cluster,  which appear before the second element of the first cluster appears at time 
$\mindex{1}{2}$.
That is, with $K(n)$ instead of $K_n$ for ease of reading:
\begin{equation}
\label{Xpdef}
\Xp  := K ( \mindex{1}{2} ). 
\end{equation}
As explained below, if $\Pi_\infty$ is exchangeable with proper random frequencies $\Pbul$ in size-biased order, then 
$\Xp$  has the same distribution as a size-biased pick from $\Pbul$:
\begin{equation}
\label{xfromk2}
\P(\Xp = k) = \E P_k  \qquad ( k = 1,2, \ldots).
\end{equation}
An extended form of this identity in distribution, giving an explicit construction from $\Pi_\infty$ of an i.i.d.\ sample of arbitrary size from the frequencies of $\Pi_\infty$ in size-biased order, 
played a key role in \cite{0gibbs}.
If $\Pbul$ is defined by the limiting cluster frequencies of $\Pi_\infty$ in their order of discovery in the restriction of $\Pi_\infty$ to $\{2,3, \ldots\}$,
then it is easily seen from Kingman's paintbox construction of $\Pi_\infty$, that $\Xp$ really is a size-biased pick from $\Pbul$:
\begin{equation}
\label{xfromke}
\P(\Xp = k \giv \Pbul) =  P_k  \qquad ( k = 1,2, \ldots)
\end{equation}
from which \eqref{xfromk2} follows by taking expectations. But if $\Pbul$ is taken to be the frequencies of clusters of $\Pi_\infty$ in their order of discovery in
$\{1,2,3, \ldots\}$,
then \eqref{xfromke} is typically false, which makes  \eqref{xfromk2} much less obvious. This gives the identity \eqref{xfromk2} a ``now you see it, now you don't'' quality.
You  see it by conditioning on the frequencies of clusters of $\Pi_\infty \cap [2,\infty)$ in their order of appearance,
but you don't see it by conditioning on the  frequencies of $\Pi_\infty$ in their usual order of least elements,

It is instructive to see exactly what is the conditional distribution of $\Xp$ given $\Pbul$, for $\Pbul$ the original frequencies of $\Pi_\infty$ in order of appearance.
To deal with non-proper frequencies let us extend the definition \eqref{Xpdef} by assuming that $X=\infty$ if the cluster $\Cl{1}=\{1\}$ in $\Pi_\infty$. 
As indicated in the Introduction, 
the conditional distribution of any exchangeable random partition $\Pi_\infty$ of $\BN$,
 given its list of cluster frequencies $\Pbul$ in order of appearance, is that of the extreme
partially exchangeable random partition of $\BN$ with the given cluster frequencies $\Pbul$. Regarding $\Pbul$ as a list of fixed frequencies $P_j \ge 0$ with $\sum_j P_j \le 1$, the
distribution of this random partition $\Pi_\infty$ is described by the {\em extreme CRP} with fixed frequencies $\Pbul$. In terms of the Chinese Restaurant metaphor in this model 
\cite[Section 3.1]{MR2245368}, 
customer $1$ sits at table $1$; thereafter, 
\begin{equation}
\label{exCRP}
\parbox{0.8\textwidth}{
given $k$ tables are occupied and there are $n_i$ customers at table $i$ for $1 \le i \le k$ with $n_1 + \cdots + n_k = n$, customer $n+1$ sits at table $i$ with probability $P_i$ for $1 \le i \le k$, and at the
new table $k+1$ with probability $1 - P_1 - \cdots - P_k$.  
}
\end{equation}
Formally, ``customer $i$ sits at table $j$'' means in present notation that $i \in \Cl{j}$.
The identity \eqref{xfromk2} now becomes the special case 
when $\Pi_\infty$ is fully exchangeable of the following description of the law of $\Xp$ given $\Pbul$
for any partially exchangeable random partition $\Pi_\infty$ of $\BN$ with limit frequencies $\Pbul$:

\begin{proposition}
\label{lmm:condplem}
Let $\Pi_\infty$ be a partially exchangeable random partition of $\BN$ with
limit frequencies $\Pbul$, and let $\Xp$ be defined as above by \eqref{Xpdef}, with $\Xp = \infty$ if $\Cl{1} = \{1\}$.
Then 
\begin{itemize} 
\item the event $(\Xp<\infty) $ equals the event $(P_1>0)$;
\item the conditional distribution of $\Xp$ given $\Pbul$ is defined by the stick-breaking formula
\begin{gather}
\label{x1dist}
\BP[ \Xp = j \giv \Pbul] = H_j \smash[b]{\prod_{i=1}^{j-1}} (1 - H_i) \mbox{ for } j = 1,2, \ldots  
\shortintertext {with}
\label{hfacts}
H_1:= P_1 \mbox{ and } H_j:= \frac{P_1}{1-P_2-P_3-\dots-P_j} \mbox{ for } j = 2,3, \ldots ;
\end{gather}
\item the unconditional probability $\BP(\Xp = j)$ is the expected value of the product in \eqref{x1dist};
\item if\/ $\Pi_\infty$ is exchangeable then $\P(\Xp = j ) = \E P_j$ for all $j = 1,2, \ldots$, meaning that $\Xp$ has the same distribution  as a size-biased pick from $\Pbul$.
\end{itemize} 
\end{proposition}
\begin{proof} The first claim follows directly from the definitions, since $\Xp=\infty$ iff $\{1\}$ is a singleton cluster of $\Pi_\infty$, which
is equivalent to $P_1=0$.
By the general theory of exchangeable and partially exchangeable random partitions recalled in the Introduction, it is enough to 
prove the formula \eqref{x1dist}
for an arbitrary fixed sequence of frequencies $\Pbul$. 
The case $j = 1$, with 
$\BP[ \Xp = 1] = P_1$, is obvious from the extreme CRP \eqref{exCRP}, because the event $(\Xp = 1)$ is identical to the event $(\mindex{1}{2} = 2)$ that the second customer is seated at table $1$.
Consider next the event 
$$
(\Xp = 2 ) = ( 2 = \mindex{2}{1} < \mindex{1}{2} < \mindex{3}{1})
$$
Conditioning on  the value $\ell$ of $\mindex{1}{2} - \mindex{2}{1}  - 1$ on this event,
the extreme CRP description \eqref{exCRP} gives
$$
\P(\Xp = 2 ) = \sum_{\ell = 0}^\infty  (1-P_1) \, P_2^\ell \, P_1 = \frac{ (1-P_1) \, P_1 }{(1-P_2)} = (1 - H_1) \, H_2 
$$
for $H_j$ as in \eqref{hfacts}.  By the same method, conditioning on values $\ell$ of $\mindex{3}{1}  - \mindex{2}{1} - 1$ and $m$ of $\mindex{1}{2}  - \mindex{3}{1} - 1$ on the
event $(\Xp = 3) = ( 2 = \mindex{2}{1} < \mindex{3}{1}  <  \mindex{1}{2} < \mindex{4}{1})$, gives
\begin{align*}
\P(\Xp = 3 ) &= \sum_{\ell = 0}^\infty  \sum_{m = 0}^\infty (1-P_1) \, P_2^\ell \, (1-P_1-P_2) ( P_2 + P_3) ^m \,P_1 \\
&= \frac{ (1-P_1) (1-P_1- P_2) P_1 }{(1-P_2) ( 1 - P_2 - P_3)}  \\
&= (1 - H_1) (1 - H_2) H_3 ,
\end{align*}
and so on. 
This gives the stick-breaking formula \eqref{x1dist}.
Taking expectations gives the unconditional distribution of $\Xp$.

The last part of the proposition is a restatement of \eqref{xfromk2}, which was explained above for the case of proper frequencies $\Pbul$.
For general case when the frequencies of an exchangeable partition may be non-proper,  it is a consequence of the more general formula \eqref{x1dist} for
partially exchangeable partitions. That $\P(\Xp = j) = \E(P_j)$ is obvious for $j = 1$. For $j = 2$ and $j=3$ this assertion becomes
\begin{align}
\label{p12} \E P_2 &= \E \, \frac{ P_1 ( 1 - P_1) }{ ( 1 - P_2) }, \\
\label{p123} \E P_3 &= \E \, \frac{ P_1 ( 1 - P_1 ) ( 1 - P_1 - P_2) }{ (1- P_2 )( 1 - P_2 - P_3) }
\end{align}
and so on. These identities are not so obvious. However, they all follow from the 
consequence  of exchangeability of $\Pi_\infty$ that for every $ k \ge 1$ such that $\P\bigl( \Pi_k = \bigl\{\{1\}, \ldots, \{k \}\bigr\}\bigr) >0$,
the joint law of $P_1, \ldots, P_k$ given this event is exchangeable \cite{MR1337249}. Consequently,  in view of \eqref{exCRP}, for every non-negative Borel measurable function
$g$ defined on $[0,1]^k$,
\begin{equation}
\label{gid}
\E g (P_1, \ldots, P_k) \prod_{i = 1}^{k-1} ( 1 - P_1 - \cdots - P_i) = \E g (P_{\sigma(1)}, \ldots, P_{\sigma(k)}) \prod_{i = 1}^{k-1} ( 1 - P_1 - \cdots - P_i).
\end{equation}
For proper frequencies this identity is known 
\cite[Theorem 4]{MR1387889}  
to characterize the collection of all possible joint distributions of frequencies $\Pbul$ of exchangeable random partitions
$\Pi_\infty$ relative to the larger class of all $\Pbul$ with $P_j \ge 0$ and $\sum_{j } P_j \le  1$ which can arise from partially exchangeable partitions.
Take $g(P_1,P_2) = P_2(1-P_1)^{-1}$ and $(\sigma(1), \sigma(2)) = (2,1)$ to recover \eqref{p12} from \eqref{gid}.
Take $g(P_1,P_2,P_3) = P_3(1-P_1)^{-1}(1- P_1 - P_2)^{-1}$ and $(\sigma(1), \sigma(2),\sigma(3)) = (2, 3, 1)$ to recover \eqref{p123} from \eqref{gid}.
For general $k$, the required evaluation of $\E P_k$ is obtained by a similar substitution in \eqref{gid} for
$g(P_1,\dots,P_k) = P_k \prod_{i = 1}^{k-1} ( 1 - P_1 - \cdots - P_i)^{-1}$ and $\sigma = (2, \ldots, k, 1)$.
\end{proof}

\begin{proof}[Proof of Corollary \ref{crl:main}]
The fact that $\bigl(\Pbuln{n}, n = 0,1, \ldots\bigr)$ is a Markov chain with stationary transition probabilities as indicated follows easily from the
description \eqref{exCRP} of the extreme CRP. If the cluster $\mathcal{C}$ containing $n$ is a singleton in $\Pi_\infty$, then $P_1=P^{(n-1)}_1=0$
and frequencies $\Pbuln{n}$ of $\Pi_{\infty}$  restricted to $\{n+1,n+2,\dots\}$ are $(P_2,P_3,\dots)$. Otherwise the cluster $\mathcal{C}$ is
infinite and it obtains a new place as in Proposition \ref{lmm:condplem}.
The rest of the corollary follows easily from the theorem and the general theory of partially exchangeable random partitions of $\BN$
presented in the introduction.
\end{proof}

The above argument places the identity
\eqref{x1dist}
in a larger context of identities comparable to \eqref{p12} and \eqref{p123}, which follow from \eqref{gid} for other choices of $\sigma$ besides the
cyclic shift. For each $k = 3,4, \ldots$ there are $k! - 2$ more such identities. For instance, for $k = 3$ there are four more expressions for $\E P_3$,
corresponding to the choices of $\sigma = (3,2,1)$, 
 $(3,1,2)$, $(1,3,2)$ and $(2,1,3)$ respectively, with varying amounts of cancellation of factors, depending on $\sigma$:
\begin{align}
\E P_3 &= \E \, \frac{ P_1 ( 1 - P_1 ) ( 1 - P_1 - P_2) }{ (1- P_3 ) ( 1 - P_2 - P_3)}  \\
&= \E \, \frac{ P_2 ( 1 - P_1 ) ( 1 - P_1 - P_2) }{ (1- P_3 ) ( 1 - P_1 - P_3)}  \\
&= \E \, \frac{ P_2 ( 1 - P_1 - P_2) }{ ( 1 - P_1 - P_3)}  \\
&= \E \, \frac{ P_3 ( 1 - P_1) }{ ( 1 - P_2)}  .
\end{align}

\section{Related literature}
\label{sec:rel}
There is a substantial literature of various models of partial exchangeability for sequences and arrays of random variables, which has been surveyed in
\cite{MR2161313}. 
The article
\cite[\S 6.2]{MR2869248} 
places the theory of partially exchangeable partitions of $\BN$ in a larger context of boundary theory for Markov chains 
evolving as a sequence of connected subsets of a directed acyclic graph that grow in the following way: initially, all vertices of the graph are unoccupied, particles are fed in one-by-one at a distinguished source vertex, 
successive particles proceed along directed edges according to an appropriate stochastic mechanism, and each particle comes to rest once it encounters an unoccupied vertex.
The article \cite{MR2744243} 
discusses questions related to the size of the first cluster in a PEP, and its interaction with other clusters.
Gnedin \cite{MR2318409} indicates an application of PEPs to records in a partially ordered set.  

\bibliographystyle{amsplain}
\bibliography{gemmax,0gibbs,0ergodic}
\end{document}